\documentclass[11pt,leqno]{amsart}
\usepackage{amsmath,amsfonts,amssymb,amscd,amsthm,amsbsy,upref}
\numberwithin{equation}{section}
\newtheorem{thm}{Theorem}[section]

\newtheorem{cor}[thm]{Corollary}

\theoremstyle{definition}
\newtheorem{defn}{Definition}[section]
\newtheorem{remark}{Remark}[section]
\newtheorem{example}{Example}[section]
\newtheorem{question}{Question}[section]

\begin{document}
\title{Perturbations of frames}

\author{Dongyang Chen, Lei Li and Bentuo Zheng}
\address{School of Mathematical Sciences\\ Xiamen University,
Xiamen,361005,China}
\email{cdy@xmu.edu.cn}
\address{School of Mathematical Sciences and LPMC, Nankai University, Tianjin, 300071, China}
\email{leilee@nankai.edu.cn}
\address{Department of Mathematical Sciences\\ The University of Memphis\\Memphis, TN 38152-3240}
\email{bzheng@memphis.edu}

\thanks{Dongyang Chen's research is supported in part by National Natural Science
Foundation of China (Grant Nos. 10701063,11001231)and the Fundamental
Research Funds for the Central Universities(Grant No.2011121039); Dongyang Chen
is a participant in the NSF Workshop in Analysis and Probability, Texas A\&M University. \\
Bentuo Zheng's Research is supported in part by NSF grant DMS-1200370;
Bentuo Zheng is a participant of the workshop at Texas A\&M University.}
\begin{abstract}
In this paper, we give some sufficient conditions under which perturbations
preserve Hilbert frames and near-Riesz bases. Similar results are also extended to
frame sequences, Riesz sequences and Schauder frames.
It is worth mentioning that some of our perturbation conditions are quite different
from those used in the previous literatures on this topic.
\end{abstract}
\date{Version: \today}
\maketitle

\baselineskip=18pt      

\section{Introduction}

Perturbation theory is a very important tool in several areas of mathematics.
It went back to the classical perturbation result by Paley and Wiener \cite{PW},
stating that a sequence that is sufficiently near to an (orthonormal)
basis in a Hilbert space automatically forms a basis. Boas \cite{B} observed that
the proof given by Paley and Wiener remains valid in an arbitrary Banach space:
\begin{thm}
Let $\{x_{n}\}_{n=1}^{\infty}$ be a basis for a Banach space
$X$ and let $\{y_{n}\}_{n=1}^{\infty}$ be a sequence in $X$.
If there exists a constant $\lambda, 0\leq \lambda<1$,
such that $$\|\sum_{i=1}^{n}c_{i}(x_{i}-y_{i})\|\leq \lambda\|\sum_{i=1}^{n}c_{i}x_{i}\|$$
for all finite sequences of scalars $c_{1},c_{2},...,c_{n}$,
then $\{y_{n}\}_{n=1}^\infty$ is also a basis for $X$.
\end{thm}
Since then, a number of variations and generalizations of this perturbation
theorem have appeared (see \cite{S}, pages.84--109).

The Paley-Wiener Theorem is so useful to show that a sequence $\{y_{n}\}_{n=1}^{\infty}$
is a Riesz basis for a Hilbert space $\mathcal{H}$ that this result is sometimes used in wavelet analysis.
But in many cases the wavelet experts prefer to work with frames, a more flexible tool,
instead of Riesz bases. Recall that a sequence $\{f_{k}\}_{k=1}^{\infty}$ in a Hilbert space $\mathcal{H}$ is a frame for $\mathcal{H}$ if there exist constants $A, B>0$ such that $$ A\|f\|^{2}\leq \sum_{k=1}^{\infty}|<f,f_{k}>|^{2}\leq B\|f\|^{2}, \forall f\in \mathcal{H}.$$  The numbers $A,B$ are called frame bounds.

Christensen strengthened Theorem 1.1 by proving the following \cite{C1}:
\begin{thm}
Let $\{f_{k}\}_{k=1}^{\infty}$ be a frame for Hilbert space $\mathcal{H}$ with bounds $A, B$
and let $\{g_{k}\}_{k=1}^{\infty}$ be a sequence in $\mathcal{H}$.
Suppose that there exist $\lambda,\mu\geq 0: \lambda+{\mu\over{\sqrt{A}}}<1$
such that $$\|\sum_{k=1}^{n}c_{k}(f_{k}-g_{k})\|\leq \lambda\|\sum_{k=1}^{n}c_{k}f_{k}\|+\mu(\sum_{k=1}^{n}|c_{k}|^{2})^{1\over{2}}$$
for all finite sequences of scalars $c_{1},c_{2},...,c_{n}$,
then $\{g_{k}\}_{k}$ is a frame for $\mathcal{H}$ with bounds
$A(1-(\lambda+{\mu\over{\sqrt{A}}}))^{2},B(1+\lambda+{\mu\over{\sqrt{B}}}))^{2}.$
\end{thm}
Note that Casazza and Christensen \cite{CC1} added a whole term on the right hand of the above inequality.
From then on, some generalizations of this result are proved for Banach frames,
atomic decompositions and Schauder frames (\cite{CC2},\cite{CH},\cite{CLL}).
In this paper, we'll give some other perturbation conditions of Hilbert frames and Schauder frames in Banach spaces.

First we give a short introduction to frame theory which we need. For more information about general frame theory, we refer the readers to \cite{C3}.

We say that a frame $\{f_{k}\}_{k=1}^{\infty}$ for $\mathcal{H}$ is semi-normalized if $$0<\inf_{1\leq k<\infty}\|f_{k}\|\leq \sup_{1\leq k<\infty}\|f_{k}\|<\infty$$
It follows from the definition that if $\{f_{k}\}_{k=1}^{\infty}$ is a frame for $\mathcal{H}$, then $\{f_{k}\}_{k=1}^{\infty}$ is complete in $\mathcal{H}$, namely, $\overline{span}\{f_{k}\}_{k=1}^{\infty}=\mathcal{H}$. We often need to consider sequences which are not complete in $\mathcal{H}$. Although they can not form frames for the whole space $\mathcal{H}$, they can form frames for their closed linear span. We say that a sequence $\{f_{k}\}_{k=1}^{\infty}$ in $\mathcal{H}$ is a frame sequence if it is a frame for its closed linear span $\overline{span}\{f_{k}\}_{k=1}^{\infty}$.

If $\{f_{k}\}_{k=1}^{\infty}$  is a frame for $\mathcal{H}$, the operator $$ T: l^{2}\rightarrow \mathcal{H}, T(\{c_{k}\}_{k=1}^{\infty})=\sum_{k=1}^{\infty}c_{k}f_{k}$$ is bounded; $T$ is called the pre-frame operator or the synthesis operator. The adjoint operator is given by $$ T^{*}: \mathcal {H}\rightarrow l^{2}, T^{*}(f)=\{<f,f_{k}>\}_{k=1}^{\infty}.$$  $T^{*}$ is called the analysis operator.  By composing $T$ and $T^{*}$, we obtain the frame operator $$ S: \mathcal{H}\rightarrow \mathcal{H}, S(f)=TT^{*}(f)=\sum_{k=1}^{\infty}<f,f_{k}>f_{k}.$$  Then $S$ is invertible and $\{S^{-1}(f_{k})\}_{k=1}^{\infty}$ is also a frame for $\mathcal{H}$ with bounds $B^{-1},A^{-1}$.  The frame  $\{f_{k}\}_{k=1}^{\infty}$ has the following important frame decomposition:
\begin{equation}
f=\sum_{k=1}^{\infty}<f,S^{-1}(f_{k})>f_{k}, \forall f\in \mathcal{H}.
\end{equation}
So $\{S^{-1}(f_{k})\}_{k=1}^{\infty}$ plays the same role in frame theory as the coefficient functionals associated to a basis. We call $\{S^{-1}(f_{k})\}_{k=1}^{\infty}$ the canonical dual frame of $\{f_{k}\}_{k=1}^{\infty}$. A frame $\{g_{k}\}_{k=1}^{\infty}$ satisfying $(1.1)$ is called a dual frame of $\{f_{k}\}_{k=1}^{\infty}$. In general, the frame $\{f_{k}\}_{k=1}^{\infty}$ and its dual frame $\{g_{k}\}_{k=1}^{\infty}$ are not biorthogonal. If a frame $\{f_{k}\}_{k=1}^{\infty}$ and its dual frame $\{g_{k}\}_{k=1}^{\infty}$ are biorthogonal, then $\{f_{k}\}_{k=1}^{\infty}$ is a Riesz basis.
We recall that a Riesz basis is a family of the form $\{U(e_{k})\}_{k=1}^{\infty}$,where $\{e_{k}\}_{k=1}^{\infty}$ is an orthonormal basis of $\mathcal{H}$ and $U: \mathcal{H}\rightarrow \mathcal{H}$ is invertible. Thus a sequence $\{f_{k}\}_{k=1}^{\infty}$ is a Riesz basis for $\mathcal{H}$ if and only if $\{f_{k}\}_{k=1}^{\infty}$ is complete in $\mathcal{H}$ and there exist two constants $A,B>0$ such that for every finite scalar sequence $\{c_{k}\}_{k=1}^{n}$, one has
\begin{equation}
A\sum_{k=1}^{n}|c_{k}|^{2}\leq \|\sum_{k=1}^{n}c_{k}f_{k}\|^{2}\leq B\sum_{k=1}^{n}|c_{k}|^{2}.
\end{equation}
$A,B$ are called lower Riesz bound,respectively,upper Riesz bound.
A sequence $\{f_{k}\}_{k=1}^{\infty}$ satisfying (1.2) for all finite scalar sequence $\{c_{k}\}_{k=1}^{n}$ is called a Riesz sequence. Thus a Riesz sequence $\{f_{k}\}_{k=1}^{\infty}$ is a Riesz basis for its closed linear span $\overline{span}\{f_{k}\}_{k=1}^{\infty}$.
If $\{f_{k}\}_{k=1}^{\infty}$ is a Riesz basis, then $\{f_{k}\}_{k=1}^{\infty}$ has the unique dual Riesz basis $\{g_{k}\}_{k=1}^{\infty}=\{S^{-1}(f_{k})\}_{k=1}^{\infty}$. Moreover, $\{f_{k}\}_{k=1}^{\infty}$ and $\{S^{-1}(f_{k})\}_{k=1}^{\infty}$ are biorthogonal.
If $\{f_{k}\}_{k=1}^{\infty}$ is a frame which is not a Riesz basis, there always exist other dual frames besides the canonical dual frame $\{S^{-1}(f_{k})\}_{k=1}^{\infty}$.
A frame $\{f_{k}\}_{k=1}^{\infty}$ for $\mathcal{H}$ is called a near-Riesz basis if it consists of a Riesz basis and a finite number of extra elements. The excess of a near-Riesz basis is defined to be the number of elements which have to be removed to obtain a Riesz basis.

Schauder frames, as a generalization of Hilbert frames to Banach spaces, were introduced in \cite{CDOSZ}.
Let $X$ be a (finite or infinite dimensional) separable Banach space. A sequence $(x_{j},f_{j})_{j\in \mathbb{J}}$ with $\{x_{j}\}_{j\in \mathbb{J}}\subset X,\{f_{j}\}_{j\in \mathbb{J}}\subset X^{*}$, and $\mathbb{J}=\mathbb{N}$ or $\mathbb{J}=\{1,2,...,N\}$ for some $N\in \mathbb{N}$, is called a Schauder frame of $X$ if for every $x\in X$, one has
\begin{equation}
x=\sum_{j\in \mathbb{J}}f_{j}(x)x_{j}.
\end{equation}
In case that $\mathbb{J}=\mathbb{N}$, we mean that the series in (1.3) converges in norm.

If the series in (1.3) converges unconditionally, we say that $(x_{j},f_{j})_{j=1}^{\infty}$ is an unconditional Schauder frame.

In general, $\{x_{j}\}_{j=1}^{\infty}$ and $\{f_{j}\}_{j=1}^{\infty}$ are not biorthogonal. If it is, then $\{x_{j}\}_{j=1}^{\infty}$ is a (Schauder) basis of $X$ and $\{f_{j}\}_{j=1}^{\infty}$ is the coefficient functionals associated to $\{x_{j}\}_{j=1}^{\infty}$. We recall that a sequence $\{x_{j}\}_{j=1}^{\infty}$ in a Banach space $X$ is called a basis if for every $x\in X$ there exists a unique sequence of scalars $\{c_{j}\}_{j=1}^{\infty}$ such that $$x=\sum_{j=1}^{\infty}c_{j}x_{j}.$$ The sequence of continuous linear functionals $\{f_{j}\}_{j=1}^{\infty}$ defined by $$f_{j}(x)=c_{j},(x=\sum_{i=1}^{\infty}c_{i}x_{i}\in X,j=1,2,...)$$
is called the sequence of coefficient functionals associated to the basis $\{x_{j}\}_{j=1}^{\infty}$. For more information about bases in Banach spaces, we refer the readers to \cite{S}.

We need some basic facts about Schauder frames and some related notations that can be found in \cite{CDOSZ}.

Let $(x_{n},f_{n})_{n=1}^{\infty}$ be a Schauder frame of a Banach space $X$ and let $Z$
be a space with a basis $\{z_{n}\}_{n=1}^{\infty}$ and corresponding coefficient
functionals $\{z^{*}_{n}\}_{n=1}^{\infty}$. We call $(Z,\{z_{n}\}_{n=1}^{\infty})$ an
associated space to $(x_{n},f_{n})_{n=1}^{\infty}$ or a sequence space associated
to $(x_{n},f_{n})_{n=1}^{\infty}$ and $\{z_{n}\}_{n=1}^{\infty}$ an associated basis, if
$$S:Z\rightarrow X, \sum_{n=1}^{\infty}c_{n}z_{n}\mapsto \sum_{n=1}^{\infty}c_{n}x_{n}$$
and $$T:X\rightarrow Z, x\mapsto \sum_{n=1}^{\infty}f_{n}(x)z_{n}$$ are
bounded operators. We call $S$ the associated reconstruction
operator and $T$ the associated decomposition operator or analysis
operator. It follows from the Uniform Bounded Principle that $$K:=\sup_{x\in B_{X}}\sup_{m\leq
n}\|\sum_{i=m}^{n}f_{i}(x)x_{i}\|<\infty .$$ We call $K$ the projection constant of
$(x_{n},f_{n})_{n=1}^{\infty}$.

Let $(x_{n},f_{n})_{n=1}^{\infty}$ be a Schauder frame of a Banach space $X$. For the sake of convenience, we may
assume that $x_{n}\neq 0(n=1,2,...)$. Define a
norm on $c_{00}$ as follows:
$$\|\sum_{i}c_{i}e_{i}\|_{Min}=\max_{m\leq
n}\|\sum_{i=m}^{n}c_{i}x_{i}\|, \ \text{ for all}\ \
\sum_{i}c_{i}e_{i}\in c_{00}.$$
Here, $c_{00}$ denotes the space of finitely supported sequences.
Denote by $Z_{Min}$ the completion
of $c_{00}$ endowed with the norm $\|\cdot\|_{Min}$. It is easy to
check that the unit vectors $\{e_{n}\}_{n=1}^{\infty}$, denoted by $\{e^{Min}_{n}\}_{n=1}^{\infty}$,
is a bi-monotone basis of $Z_{Min}$. We call $Z_{Min}$ and
$\{e^{Min}_{n}\}_{n=1}^{\infty}$ the minimal space and the minimal basis with
respect to $(x_{n},f_{n})_{n=1}^{\infty}$. Note that the operator
$$S_{Min}: Z_{Min}\rightarrow X, \sum_{n=1}^{\infty}c_{n}e^{Min}_{n}\mapsto
\sum_{n=1}^{\infty}c_{n}x_{n}$$ is well-defined, linear and bounded with
$\|S_{Min}\|=1$.

And the operator $$T_{Min}: X\rightarrow Z_{Min}, x\mapsto
\sum_{n=1}^{\infty}f_{n}(x)e^{Min}_{n}$$ is well-defined,linear, and bounded
with $\|T_{Min}\|\leq K$.

Just as the near-Riesz bases in Hilbert spaces, we say that a Schauder frame $(x_{n},f_{n})_{n=1}^{\infty}$ of a Banach space $X$ is a near-Schauder basis if $\{x_{n}\}_{n=1}^{\infty}$ consists of a Schauder basis and a finite number of extra elements.
The excess of a near-Schauder basis is defined to be the number of elements which have to be removed to obtain a Schauder basis.

\section{Perturbations of frames in Hilbert spaces}

We say that two sequences $\{f_{k}\}_{k=1}^{\infty}$ and $\{g_{k}\}_{k=1}^{\infty}$ in
$\mathcal{H}$ are quadratically close if
$\sum_{k=1}^{\infty}\|f_{k}-g_{k}\|^{2}<\infty$. Christensen \cite{C2} showed that if
$\{f_{k}\}_{k=1}^{\infty}$ is a frame for $\mathcal{H}$ and $\{g_{k}\}_{k=1}^{\infty}$ is a
sequence which is quadratically close to $\{f_{k}\}_{k=1}^{\infty}$, then
$\{g_{k}\}_{k=1}^{\infty}$ is a frame for
${\overline{span}}\{g_{k}\}_{k=1}^{\infty}$.
In general, the $\{g_{k}\}_{k=1}^{\infty}$ is not a frame for the whole space $\mathcal{H}$. For example,
if $\{f_{k}\}_{k=1}^{\infty}$ is an orthonormal basis in $\mathcal{H}$
and if we define $g_{1}=0,g_{k}=f_{k}(k\geq 2$), then $\sum_{k=1}^{\infty}\|f_{k}-g_{k}\|^{2}=1$; but $\{g_{k}\}_{k=1}^{\infty}$ is not complete in $\mathcal{H}$.

In Banach space theory, it is well-known (Theorem 10.1 in \cite{S}) that if $\{x_{n}\}_{n=1}^{\infty}$ is a (Schauder)
basis in a Banach space $X$ and $\{y_{n}\}_{n=1}^{\infty}$ is a sequence
which  satisfies $\sum_{n=1}^{\infty}\|x_{n}-y_{n}\|\|f_{n}\|<1$, then
$\{y_{n}\}_{n=1}^{\infty}$ is a basis for $X$, where $\{f_{n}\}_{n=1}^{\infty}$ is the
coefficient functionals to $\{x_{n}\}_{n=1}^{\infty}$.

By adding a condition (similar to the above Banach space condition) to Christensen's assumption, we are able to show that perturbations under these conditions are indeed frames for the whole space. Moreover we get estimates for the lower and upper frame bounds of the perturbations.

\begin{thm}\label{thm21}
Let $\{f_{k}\}_{k=1}^{\infty}$ be a frame for $H$ with bounds $A,B$ and $\{g_{k}\}_{k=1}^{\infty}$ be the dual frame of $\{f_{k}\}_{k=1}^{\infty}$ with bounds $C,D$.
Assume that $\{h_{k}\}_{k=1}^{\infty}$ is a sequence in $\mathcal{H}$ which satisfies the
following two conditions:
\item[(1)]$\lambda:=\sum_{k=1}^{\infty}\|f_{k}-h_{k}\|^{2}<\infty$;
\item[(2)]$\mu:=\sum_{k=1}^{\infty}\|f_{k}-h_{k}\|\|g_{k}\|<1$.

Then $\{h_{k}\}_{k=1}^{\infty}$ is a frame for $\mathcal{H}$ with bounds ${1\over{D}}(1-\mu)^{2}$,
$B(1+{\sqrt{\lambda}\over\sqrt{B}})^{2}$.
\end{thm}

\begin{proof}
Let $T:l^{2}\rightarrow \mathcal{H}$ be the pre-frame operator of
$\{f_{k}\}_{k=1}^{\infty}$. Then $\|T\|\leq \sqrt{B}$.

Define
$$U:l^{2}\rightarrow \mathcal{H}, U((c_{k})_{k=1}^{\infty})=\sum_{k=1}^{\infty}c_{k}h_{k}.$$
 (1) implies that $U$ is well-defined and $\|U\|\leq
\sqrt{\lambda}+\sqrt{B}$.

Then $$U^{*}(f)=\{<f,h_{k}>\}_{k=1}^{\infty}, \forall f\in \mathcal{H}.$$ So
\begin{eqnarray*}
\sum_{k=1}^{\infty}|<f,h_{k}>|^{2}&=&\|U^{*}(f)\|^{2}\\
&\leq & B(1+{\sqrt{\lambda}\over\sqrt{B}})^{2}\|f\|^{2}.
\end{eqnarray*}

Define $$L: \mathcal{H}\rightarrow \mathcal{H},
L(f)=\sum_{k=1}^{\infty}<f,g_{k}>h_{k}.$$

By (2), for any $f\in \mathcal{H}$, we have
\begin{eqnarray*}
\|f-L(f)\|&=&\|\sum_{k=1}^{\infty}<f,g_{k}>(f_{k}-h_{k})\|\\
&\leq &\sum_{k=1}^{\infty}\|f_{k}-h_{k}\|\|g_{k}\|\|f\|\\
&=&\mu\|f\|.
\end{eqnarray*}
Thus $L$ is an isomorphism from $\mathcal{H}$ onto $\mathcal{H}$ with
$\|L^{-1}\|\leq {1\over{1-\mu}}$.

Every $f\in \mathcal{H}$ can be written as
$$f=LL^{-1}(f)=\sum_{k=1}^{\infty}<L^{-1}f,g_{k}>h_{k}.$$
This leads to
\begin{eqnarray*}
\|f\|^{2}=<f,f>&=&<\sum_{k=1}^{\infty}<L^{-1}f,g_{k}>h_{k},f>\\
&=&\sum_{k=1}^{\infty}<L^{-1}f,g_{k}><h_{k},f>\\
&\leq &(\sum_{k=1}^{\infty}|<L^{-1}f,g_{k}>|^{2})^{1\over{2}}(\sum_{k=1}^{\infty}|<h_{k},f>|^{2})^{1\over{2}}\\
&\leq &\sqrt{D}\cdot \|L^{-1}f\|(\sum_{k=1}^{\infty}|<h_{k},f>|^{2})^{1\over{2}}\\
&\leq
&\sqrt{D}\cdot{1\over{1-\mu}}\|f\|(\sum_{k=1}^{\infty}|<h_{k},f>|^{2})^{1\over{2}}.
\end{eqnarray*}
Therefore $${1\over{D}}(1-\mu)^{2}\|f\|^{2}\leq \sum_{n}|<h_{k},f>|^{2}.$$

This completes the proof.

\end{proof}

\begin{remark}
In Theorem 2.1, if the dual frame $\{g_{k}\}_{k=1}^{\infty}$ is semi-normalized,
then condition (2) would imply condition (1). However this is not true in general.
In the following example, we'll give three frames: $\{f_{k}\}_{k=1}^{\infty}$, its dual frame $\{g_{k}\}_{k=1}^{\infty}$ and $\{h_{k}\}_{k=1}^{\infty}$.
These three frames satisfy condition (2), but do not satisfy condition (1).

\end{remark}

\begin{example}\label{ex21}
Let $\{e_{n}\}_{n=1}^{\infty}$ be an orthonormal basis for $\mathcal{H}$.
Choose an integer $N$ with $\sum_{n=N+1}^{\infty}{1\over{n^{2}}}<1$.
Let $\{c_{n}\}_{n=1}^{\infty}=\{(n+N)^{2}\}_{n=1}^{\infty},
\{t_{n}\}_{n=1}^{\infty}=\{n^{3\over{2}}+1\}_{n=1}^{\infty}$ and $\{k_{n}\}_{n=1}^{\infty}=\{c_{n}^{2}\}_{n=1}^{\infty}$.
Let $\{f_{k}\}_{k=1}^{\infty}$ be the sequence where each ${1\over {c_{n}}}e_{n}$ is repeated by
$k_{n}$ times. That is,

$$\{f_{k}\}_{k=1}^{\infty}: \underbrace{{1\over {c_{1}}}e_{1},{1\over {c_{1}}}e_{1},...,{1\over {c_{1}}}e_{1}}_{k_{1}},...,\underbrace{{1\over {c_{n}}}e_{n},{1\over {c_{n}}}e_{n},...,{1\over {c_{n}}}e_{n}}_{k_{n}},...$$

Let

$$\{h_{k}\}_{k=1}^{\infty}: \underbrace{{t_{1}\over {c_{1}}}e_{1},{1\over {c_{1}}}e_{1},...,{1\over {c_{1}}}e_{1}}_{k_{1}},...,\underbrace{{t_{n}\over {c_{n}}}e_{n},{1\over {c_{n}}}e_{n},...,{1\over {c_{n}}}e_{n}}_{k_{n}},...$$

Then, for each $f\in \mathcal{H}$, we have

\begin{align*}
\sum_{k=1}^{\infty}|<f,f_{k}>|^{2}&=\sum_{n=1}^{\infty}k_{n}|<f,{1\over {c_{n}}}e_{n}>|^{2}\\
&=\sum_{n=1}^{\infty}|<f,e_{n}>|^{2}\\
&=\|f\|^{2}.
\end{align*}

Thus $\{f_{k}\}_{k=1}^{\infty}$ is a tight frame with frame bound $A=B=1$ and hence the frame operator $S$ is the identity operator $I_{\mathcal{H}}$.
We take $\{g_{k}\}_{k=1}^{\infty}=\{S^{-1}(f_{k})\}_{k=1}^{\infty}=\{f_{k}\}_{k=1}^{\infty}$.
Next we show that $\{h_{k}\}_{k=1}^{\infty}$ is also a frame for $\mathcal{H}$. Indeed, for any $f\in \mathcal{H}$, one has

\begin{align*}
\sum_{k=1}^{\infty}|<f,h_{k}>|^{2}&=\sum_{n=1}^{\infty}({{t_{n}^{2}-1}\over{c_{n}^{2}}}|<f,e_{n}>|^{2}+{k_{n}\over{c_{n}^{2}}}|<f,e_{n}>|^{2})\\
&=\sum_{n=1}^{\infty}{{t_{n}^{2}-1}\over{c_{n}^{2}}}|<f,e_{n}>|^{2}+\|f\|^{2}\\
&\leq \sum_{n=1}^{\infty}|<f,e_{n}>|^{2}+\|f\|^{2}\\
&=2\|f\|^{2}.
\end{align*}

Thus $$\|f\|^{2}\leq \sum_{k=1}^{\infty}|<f,h_{k}>|^{2}\leq 2\|f\|^{2}.$$

But

\begin{align*}
\sum_{k=1}^{\infty}\|f_{k}-h_{k}\|^{2}&=\sum_{n=1}^{\infty}|{{t_{n}-1}\over{c_{n}}}|^{2}\\
&=\sum_{n=1}^{\infty}{n^{3}\over{(n+N)^{4}}}\\
&=\infty.
\end{align*}
and
\begin{align*}
\sum_{k=1}^{\infty}\|f_{k}-h_{k}\|\|g_{k}\|&=\sum_{k=1}^{\infty}\|f_{k}-h_{k}\|\|f_{k}\|\\
&=\sum_{n=1}^{\infty}{{|t_{n}-1}|\over{c_{n}^{2}}}\\
&=\sum_{n=1}^{\infty}{n^{3\over{2}}\over{(n+N)^{4}}}\\
&\leq \sum_{n=1}^{\infty}{1\over{(n+N)^{2}}}\\
&<1.
\end{align*}

\end{example}

\begin{remark}
We can also construct an example for which Theorem \ref{thm21} works while \cite[Theorem 1]{C2} does not. Actually, there are many such examples.
\end{remark}
\begin{example}
Let $\{e_{n}\}_{n=1}^{\infty}$ be an orthonormal basis for $\mathcal{H}$.
Let $\{c_{n}\}_{n=1}^{\infty}=\{n+1\}_{n=1}^{\infty},  t_1=3$ and
$t_{n}=2$ for all $n\geq 2$ and $\{k_{n}\}_{n=1}^{\infty}=\{c_{n}^{2}\}_{n=1}^{\infty}$.
Let $\{f_{k}\}_{k=1}^{\infty}$ be the sequence where each ${1\over {c_{n}}}e_{n}$ is repeated by
$k_{n}$ times. That is,

$$\{f_{k}\}_{k=1}^{\infty}: \underbrace{{1\over {c_{1}}}e_{1},{1\over {c_{1}}}e_{1},...,{1\over {c_{1}}}e_{1}}_{k_1},...,\underbrace{{1\over {c_{n}}}e_{n},{1\over {c_{n}}}e_{n},...,{1\over {c_{n}}}e_{n}}_{k_n},...$$

Let

$$\{h_{k}\}_{k=1}^{\infty}: \underbrace{{t_{1}\over {c_{1}}}e_{1},{1\over {c_{1}}}e_{1},...,{1\over {c_{1}}}e_{1}}_{k_1},...,\underbrace{{t_{n}\over {c_{n}}}e_{n},{1\over {c_{n}}}e_{n},...,{1\over {c_{n}}}e_{n}}_{k_n},...$$

An argument similar to Example \ref{ex21} shows that $\{f_k\}_{k=1}^\infty$
is a tight frame with bounds $A=B=1$. Hence the frame operator $S$ is the identity operator $I_{\mathcal{H}}$.
We take $\{g_{k}\}_{k=1}^{\infty}=\{S^{-1}(f_{k})\}_{k=1}^{\infty}=\{f_{k}\}_{k=1}^{\infty}$.
Moreover, we can derive that
\[\|f\|^2\leq \sum_{k=1}^\infty |<f, h_k>|^2\leq 3\|f\|^2, \forall f\in \mathcal{H}.\] So
$\{h_k\}_{k=1}^\infty$ is also a frame for $\mathcal{H}$.
Note that \[\infty>\sum_{k=1}^\infty \|f_k-h_k\|^2=\sum_{n=1}^\infty \frac{(t_n-1)^2}{c_n^2}=1+\sum_{n=2}^\infty\frac{1}{(n+1)^2}>1,\]
Thus \cite[Theorem 1]{C2} does not work for $\{f_k\}_{k=1}^\infty$ and $\{h_k\}_{k=1}^{\infty}$.
On the other hand, we can derive that
\begin{eqnarray*}
\sum_{k=1}^\infty \|f_k-h_k\| \|g_k\|&=& \sum_{n=1}^\infty \frac{t_n-1}{c_n}\frac{1}{c_n}\\
&=&\frac{1}{2}+\sum_{n=2}^\infty\frac{1}{(n+1)^2}<1,
\end{eqnarray*}
and then Theorem \ref{thm21} works for $\{f_{k}\}_{k=1}^{\infty}, \{g_{k}\}_{k=1}^{\infty}$ and $\{h_{k}\}_{k=1}^{\infty}$.
\end{example}

\begin{remark}
Favier and Zalik \cite[Theorem 3]{FZ95}  showed that if $\{f_{k}\}_{k=1}^{\infty}$ is a frame for $\mathcal{H}$ with bounds $A$ and $B$,
and $\{h_{k}\}_{k=1}^{\infty}$ is a sequence in $\mathcal{H}$ such that
$\{f_{k}-h_{k}\}_{k=1}^{\infty}$ is a Bessel sequence with bound $M<A$,
then $\{h_{k}\}_{k=1}^{\infty}$ is a frame for $\mathcal{H}$ with bounds $(1-{\sqrt{\frac{M}{A}}})^{2}A$,
$(1+\sqrt{\frac{M}{B}})^{2}B$.
In \cite[Theorem 3]{FZ95} , the bound $M$ of the Bessel sequence $\{f_{k}-h_{k}\}_{k=1}^{\infty}$ needs to be $<A$, but in Theorem \ref{thm21} , the Bessel sequence's bound is $\lambda$ and $\lambda<\infty$(not $\lambda<A$).

\end{remark}

In the case that $\{f_{k}\}_{k=1}^{\infty}$ is a near-Riesz basis for $\mathcal{H}$, condition (2) in Theorem 2.1 implies condition (1). So the following corollary is stated with condition (2) only.

\begin{cor}
Let $\{f_{k}\}_{k=1}^{\infty}$ be a near-Riesz basis for $\mathcal{H}$. If $\{h_{k}\}_{k=1}^{\infty}$ is a sequence in  $\mathcal{H}$ which satisfies $\mu:=\sum_{k=1}^{\infty}\|f_{k}-h_{k}\|\|S^{-1}(f_{k})\|<1$, then $\{h_{k}\}_{k=1}^{\infty}$ is a near-Riesz basis for $\mathcal{H}$; In this case, $\{f_{k}\}_{k=1}^{\infty}$ and $\{h_{k}\}_{k=1}^{\infty}$ has the same excess.

\end{cor}

\begin{proof}
Suppose that $\{f_{k}\}_{k=1}^{\infty}$ is a near-Riesz basis with excess $n$. We may assume that, by changing the index set,$\{f_{k}\}_{k=1}^{\infty}=\{f_{k}\}_{k=1}^{n}\cup \{f_{k}\}_{k=n+1}^{\infty}$,where $\{f_{k}\}_{k=n+1}^{\infty}$ is a Riesz basis for $\mathcal{H}$. Then there exist two constants $A,B>0$ such that $$A\sum_{k=n+1}^{l}|c_{k}|^{2}\leq \|\sum_{k=n+1}^{l}c_{k}f_{k}\|^{2}\leq B\sum_{k=n+1}^{l}|c_{k}|^{2}$$ for all finite scalar sequences $\{c_{k}\}_{k=n+1}^{l},l=n+1,n+2,...$.

Since  $\{f_{k}\}_{k=n+1}^{\infty}$ is a Riesz basis for $\mathcal{H}$, we have $$0<\inf_{k\geq n+1}\|f_{k}\|\leq \sup_{k\geq n+1}\|f_{k}\|<\infty.$$  It follows from the assumption that $$\sum_{k=n+1}^{\infty}\|f_{k}-h_{k}\|<\infty.$$  It is immediate that $$\sum_{k=1}^{\infty}\|f_{k}-h_{k}\|^{2}<\infty.$$ It follows from Theorem 2.1 that $\{h_{k}\}_{k=1}^{\infty}$ is a frame for $\mathcal{H}$ and hence is complete in $\mathcal{H}$.

We choose $m>n$ such that $\sum_{k=m+1}^{\infty}\|f_{k}-h_{k}\|^{2}<A$.

Define $$L: \overline{span}\{f_{k}\}_{k=m+1}^{\infty}\rightarrow \overline{span}\{h_{k}\}_{k=m+1}^{\infty}, L(\sum_{k=m+1}^{\infty}c_{k}f_{k})=\sum_{k=m+1}^{\infty}c_{k}h_{k}.$$
Then $L$ is well-defined, bounded and $L(f_{k})=h_{k}$ for all $k\geq m+1$.

Moreover, for any element $\sum_{k=m+1}^{\infty}c_{k}f_{k}\in \overline{span}\{f_{k}\}_{k=m+1}^{\infty}$, we have
\begin{align*}
\|\sum_{k=m+1}^{\infty}c_{k}f_{k}-L(\sum_{k=m+1}^{\infty}c_{k}f_{k})\|&=\|\sum_{k=m+1}^{\infty}c_{k}f_{k}-\sum_{k=m+1}^{\infty}c_{k}h_{k}\|\\
&\leq \sum_{k=m+1}^{\infty}|c_{k}|\|f_{k}-h_{k}\|\\
&\leq (\sum_{k=m+1}^{\infty}|c_{k}|^{2})^{1\over 2}(\sum_{k=m+1}^{\infty}\|f_{k}-h_{k}\|^{2})^{1\over 2}\\
&\leq {1\over \sqrt{A}}\|\sum_{k=m+1}^{\infty}c_{k}f_{k}\|(\sum_{k=m+1}^{\infty}\|f_{k}-h_{k}\|^{2})^{1\over 2}.
\end{align*}
Therefore $L(\overline{span}\{f_{k}\}_{k=m+1}^{\infty})$ is closed.

Consequently, $L(\overline{span}\{f_{k}\}_{k=m+1}^{\infty})=\overline{span}\{h_{k}\}_{k=m+1}^{\infty}$
and $\{h_{k}\}_{k=m+1}^{\infty}$ is a Riesz basic sequence.

It follows from Theorem 8 in \cite{CK} that $$dim \mathcal{H}/\overline{span}\{f_{k}\}_{k=m+1}^{\infty}=dim \mathcal{H}/\overline{span}\{h_{k}\}_{k=m+1}^{\infty}.$$
Since $\{h_{k}\}_{k=1}^{\infty}$ is complete in $\mathcal{H}$,
we can pick the same numbers of linearly independent elements from $\{f_{k}\}_{k=1}^{m}$(resp.$\{h_{k}\}_{k=1}^{m}$) to add to $\{f_{k}\}_{k=m+1}^{\infty}$(
resp.$\{h_{k}\}_{k=m+1}^{\infty}$) to get Riesz bases for $\mathcal{H}$. We conclude that $\{f_{k}\}_{k=1}^{\infty}$ and $\{h_{k}\}_{k=1}^{\infty}$ has the same excess.

\end{proof}

As mentioned above, if $\{f_{k}\}_{k=1}^{\infty}$ is a frame for $\mathcal{H}$ and $\{g_{k}\}_{k=1}^{\infty}$ is a
sequence which is quadratically close to $\{f_{k}\}_{k=1}^{\infty}$, then $\{g_{k}\}_{k=1}^{\infty}$ is a frame sequence. A natural question is:

\begin{question}
If $\{f_{k}\}_{k=1}^{\infty}$ is a frame sequence in $\mathcal{H}$ and $\{g_{k}\}_{k=1}^{\infty}$ is a
sequence in $\mathcal{H}$ which is quadratically close to $\{f_{k}\}_{k=1}^{\infty}$, is $\{g_{k}\}_{k=1}^{\infty}$ a frame sequence?
\end{question}

The following example shows that this question is false.
\begin{example}
Let $\{e_{k}\}_{k=1}^{\infty}$ be an orthonormal basis for $\mathcal{H}$. Let
$$\{f_{k}\}_{k=1}^{\infty}: 0,     e_{2},       0,            e_{4},       0,           e_{6},...$$
and
$$\{g_{k}\}_{k=1}^{\infty}: e_{1}, e_{2},  {1\over{3}}e_{3},  e_{4}, {1\over{5}}e_{5},  e_{6},...$$

Obviously, $\{f_{k}\}_{k=1}^{\infty}$ is a Riesz sequence and $$\sum_{k=1}^{\infty}\|f_{k}-g_{k}\|^{2}=\sum_{n=1}^{\infty}({1\over{2n-1}})^{2}<\infty.$$
But $\{g_{k}\}_{k=1}^{\infty}$ is not a frame sequence. Indeed, for each $n$, one has $$\sum_{k=1}^{\infty}|<{1\over {2n-1}}e_{2n-1},g_{k}>|^{2}={1\over (2n-1)^{4}}.$$
but $$\|{1\over {2n-1}}e_{2n-1}\|^{2}={1\over (2n-1)^{2}}.$$
\end{example}
We observe that $(\overline{span}\{e_{2k}\}_{k=1}^{\infty})^{\perp}$ is infinite dimensional in the  Example 2.2. Actually, we'll find that if $\{f_{k}\}_{k=1}^{\infty}$ is a frame sequence in $\mathcal{H}$ with $dim(\overline{span}\{f_{k}\}_{k=1}^{\infty})^{\perp}=\infty$, there exist a frame sequence $\{g_{k}\}_{k=1}^{\infty}$,one of whose subsequences is $\{f_{k}\}_{k=1}^{\infty}$,
and a sequence $\{h_{k}\}_{k=1}^{\infty}$ which is quadratically close to $\{g_{k}\}_{k=1}^{\infty}$, but $\{h_{k}\}_{k=1}^{\infty}$ is not a frame sequence.

Before giving our result, we need a definition.

\begin{defn}
Let $\{f_{k}\}_{k=1}^{\infty}$ be a frame sequence in $\mathcal{H}$. We say that a sequence $\{g_{k}\}_{k=1}^{\infty}$ in $\mathcal{H}$ is a frame extension of
$\{f_{k}\}_{k=1}^{\infty}$ if it is a frame sequence and $\{f_{k}\}_{k=1}^{\infty}$ is a subsequence of it.
\end{defn}

\begin{thm}
Let $\{f_{k}\}_{k=1}^{\infty}$ be a frame sequence in $\mathcal{H}$. The following are equivalent:
\item[(1)]$(\overline{span}\{f_{k}\}_{k=1}^{\infty})^{\perp}$ is finite dimensional;
\item[(2)]For any frame extension $\{g_{k}\}_{k=1}^{\infty}$ of $\{f_{k}\}_{k=1}^{\infty}$ and for any sequence $\{h_{k}\}_{k=1}^{\infty}$ which is quadratically close to $\{g_{k}\}_{k=1}^{\infty}$, then $\{h_{k}\}_{k=1}^{\infty}$ is a frame sequence.
\end{thm}
\begin{proof}
$(1)\Rightarrow (2):$ Given any frame extension $\{g_{k}\}_{k=1}^{\infty}$ of $\{f_{k}\}_{k=1}^{\infty}$ and any sequence $\{h_{k}\}_{k=1}^{\infty}$ which is quadratically close to $\{g_{k}\}_{k=1}^{\infty}$.
Since $(\overline{span}\{f_{k}\}_{k=1}^{\infty})^{\perp}$ is finite dimensional, it follows that
$(\overline{span}\{g_{k}\}_{k=1}^{\infty})^{\perp}$ is also finite dimensional.
Suppose that $dim(\overline{span}\{g_{k}\}_{k=1}^{\infty})^{\perp}=n$. Take an orthonormal basis $\{e_{k}\}_{k=1}^{n}$ for $(\overline{span}\{g_{k}\}_{k=1}^{\infty})^{\perp}$. Then $\{e_{k}\}_{k=1}^{n}\cup \{g_{k}\}_{k=1}^{\infty}$ is a frame for $\mathcal{H}$. Since $\{h_{k}\}_{k=1}^{\infty}$ is quadratically close to $\{g_{k}\}_{k=1}^{\infty}$, $\{e_{k}\}_{k=1}^{n}\cup\{h_{k}\}_{k=1}^{\infty}$ is also quadratically close to $\{e_{k}\}_{k=1}^{n}\cup\{g_{k}\}_{k=1}^{\infty}$. Theorem 3 in \cite{C2} yields that $\{e_{k}\}_{k=1}^{n}\cup\{h_{k}\}_{k=1}^{\infty}$ is a frame for $\overline{span}\{\{e_{k}\}_{k=1}^{n}\cup\{h_{k}\}_{k=1}^{\infty}\}$. It follows from Lemma 2 in \cite{C2} that $\{h_{k}\}_{k=1}^{\infty}$ is a frame sequence.

$(2)\Rightarrow (1):$ Suppose that $(\overline{span}\{f_{k}\}_{k=1}^{\infty})^{\perp}$ is infinite dimensional. Take an orthonormal basis $\{e_{k}\}_{k=1}^{\infty}$ for $(\overline{span}\{f_{k}\}_{k=1}^{\infty})^{\perp}$.

Let $$\{g_{k}\}_{k=1}^{\infty}: 0, f_{1}, 0, f_{2}, 0, f_{3},...$$
and $$\{h_{k}\}_{k=1}^{\infty}: e_{1}, f_{1}, {1\over{3}}e_{3}, f_{2}, {1\over{5}}e_{5}, f_{3},...$$

Then $\{g_{k}\}_{k=1}^{\infty}$ is a frame extension of $\{f_{k}\}_{k=1}^{\infty}$ and
$$\sum_{k=1}^{\infty}\|g_{k}-h_{k}\|^{2}=\sum_{n=1}^{\infty}({1\over{2n-1}})^{2}<\infty.$$
But $\{h_{k}\}_{k=1}^{\infty}$ is not a frame sequence. Indeed, as in the Example 2.2, for each $n$, one has $$\sum_{k=1}^{\infty}|<{1\over {2n-1}}e_{2n-1},h_{k}>|^{2}={1\over (2n-1)^{4}},$$
but $$\|{1\over {2n-1}}e_{2n-1}\|^{2}={1\over (2n-1)^{2}}$$
from which the result follows.

\end{proof}

So under what conditions is Question 2.1 true? Inspired by the paper of Christensen, Lennard and Lewis (\cite{CLL}), we'll use the gap to study Question 2.1. First we recall the definition of gap that can be found in \cite{K}.

\begin{defn}
Let $K$ and $L$ be subspaces of $\mathcal{H}$. When $K\neq \{0\}$, the gap from $K$ to $L$ is given by $$\delta(K,L):=\sup_{x\in K,\|x\|=1}\inf_{y\in L}\|x-y\|.$$
Also, when $K=\{0\}$, we define $\delta(K,L)=0$.

\end{defn}

Now we can state our result which is the extension of Theorem 2.1.

\begin{thm}
Let $\{f_{k}\}_{k=1}^{\infty}$ be a frame sequence in $\mathcal{H}$ with bounds $A,B$ and let $\{g_{k}\}_{k=1}^{\infty}$ be the dual frame sequence of $\{f_{k}\}_{k=1}^{\infty}$ with bounds $C,D$. Suppose that $\{h_{k}\}_{k=1}^{\infty}$ is a sequence in $\mathcal{H}$.
Let $K:=\overline{span}\{h_{k}\}_{k=1}^{\infty}, L:=\overline{span}\{f_{k}\}_{k=1}^{\infty}$. Assume that $\delta(K,L)<1$. If $\{h_{k}\}_{k=1}^{\infty}$ satisfies the following two conditions:
\item[(1)]$\lambda:=\sum_{k=1}^{\infty}\|f_{k}-h_{k}\|^{2}<\infty$;
\item[(2)]$\mu:=\sum_{k=1}^{\infty}\|f_{k}-h_{k}\|\|g_{k}\|<1$.

Then $\{h_{k}\}_{k=1}^{\infty}$ is a frame sequence with bounds ${1\over{D}}(1-\mu)^{2}$,
$B(1+{\sqrt{\lambda}\over\sqrt{B}})^{2}{1\over{(1-\delta(K,L))^{2}}}$. Moreover, the restriction of the orthogonal projection $P_{L}$ to $K$ is an isomorphism from $K$ onto $L$.
\end{thm}

\begin{proof}
By the assumption $\delta(K,L)<1$, we have, for any $h\in K$,
$$\|P_{L}(h)\|\geq (1-\delta(K,L))\|h\|$$
Since $\sum_{k=1}^{\infty}\|f_{k}-P_{L}(h_{k})\|^{2}\leq \sum_{k=1}^{\infty}\|f_{k}-h_{k}\|^{2}$ and $\sum_{k=1}^{\infty}\|f_{k}-P_{L}(h_{k})\|\|g_{k}\|\leq \sum_{k=1}^{\infty}\|f_{k}-h_{k}\|\|g_{k}\|$,
we apply Theorem 2.1 to the sequence $\{P_{L}(h_{k})\}_{k=1}^{\infty}$ in the Hilbert space $L$ and the frame $\{f_{k}\}_{k=1}^{\infty}$ for $L$ to obtain that $\{P_{L}(h_{k})\}_{k=1}^{\infty}$ is a frame for $L$ with bounds ${1\over{D}}(1-\mu)^{2}$,
$B(1+{\sqrt{\lambda}\over\sqrt{B}})^{2}$. In particular, $P_{L}(K)=L$ and hence $Q:=P_{L}|_{K}$ is an isomorphism from $K$ onto $L$. It remains to show that $\{h_{k}\}_{k=1}^{\infty}$ is a frame for $K$. Indeed, for $h\in K$, we have
\begin{align*}
\sum_{k=1}^{\infty}|<h,h_{k}>|^{2}&=\sum_{k=1}^{\infty}|<h,Q^{-1}Q(h_{k})>|^{2}\\
&=\sum_{k=1}^{\infty}|<(Q^{-1})^{*}h,Q(h_{k})>|^{2}\\
&\leq B(1+{\sqrt{\lambda}\over\sqrt{B}})^{2}\|(Q^{-1})^{*}h\|^{2}\\
&\leq B(1+{\sqrt{\lambda}\over\sqrt{B}})^{2}\|Q^{-1}\|^{2}\|h\|^{2}\\
&\leq B(1+{\sqrt{\lambda}\over\sqrt{B}})^{2}{1\over{(1-\delta(K,L))^{2}}}\|h\|^{2}.
\end{align*}
On the other hand,
\begin{align*}
\sum_{k=1}^{\infty}|<h,h_{k}>|^{2}&=\sum_{k=1}^{\infty}|<(Q^{-1})^{*}h,Q(h_{k})>|^{2}\\
&\geq {1\over{D}}(1-\mu)^{2}\|(Q^{-1})^{*}h\|^{2}\\
&={1\over{D}}(1-\mu)^{2}\|(Q^{*})^{-1}h\|^{2}\\
&\geq {1\over{D}}(1-\mu)^{2}\|h\|^{2}.
\end{align*}

The proof is completed.

\end{proof}

In the special case that $\{f_{k}\}_{k=1}^{\infty}$ is a Riesz sequence, the gap does not need to be included in the above theorem.

\begin{thm}
Let $\{f_{k}\}_{k=1}^{\infty}$ be a Riesz sequence in $\mathcal{H}$ with bounds $A,B$ and
let $\{g_{k}\}_{k=1}^{\infty}$ be a sequence in $\mathcal{H}$ which satisfies $\mu:=\sum_{k=1}^{\infty}\|f_{k}-g_{k}\|\|S^{-1}(f_{k})\|<1$.
Then $\{g_{k}\}_{k=1}^{\infty}$ is a Riesz sequence with bounds $A(1-\mu)^{2},B(1+{\sqrt{\lambda}\over\sqrt{B}})^{2}$,where $\lambda:=\sum_{k=1}^{\infty}\|f_{k}-g_{k}\|^{2}$.
\end{thm}

\begin{proof}
It is easy to show that $$\|\sum_{k=1}^{\infty}c_{k}g_{k}\|\leq (\sqrt{B}+\sqrt{\lambda})(\sum_{k=1}^{\infty}|c_{k}|^{2})^{1\over{2}},
\forall \{c_{k}\}_{k=1}^{\infty}\in l^{2}.$$ Define $$U:\mathcal{H}\rightarrow \mathcal{H}, U(f)=\sum_{k=1}^{\infty}<f,S^{-1}f_{k}>g_{k}.$$ Then $U$ is well-defined,linear and bounded with
$\|U\|\leq {{\sqrt{B}+\sqrt{\lambda}}\over{\sqrt{A}}}$.

Indeed, for any $f\in \mathcal{H}$, we have
\begin{align*}
\|\sum_{k=1}^{\infty}<f,S^{-1}f_{k}>g_{k}\|&\leq (\sqrt{B}+\sqrt{\lambda})(\sum_{k=1}^{\infty}|<f,S^{-1}f_{k}>|^{2})^{1\over{2}}\\
&=(\sqrt{B}+\sqrt{\lambda})(\sum_{k=1}^{\infty}|<P_{L}f,S^{-1}f_{k}>|^{2})^{1\over{2}}\\
&\leq (\sqrt{B}+\sqrt{\lambda})\cdot({1\over{A}}\|P_{L}f\|^{2})^{1\over{2}}\\
&\leq {{\sqrt{B}+\sqrt{\lambda}}\over{\sqrt{A}}}\|f\|.
\end{align*}
where $L:=\overline{span}\{f_{k}\}_{k=1}^{\infty}$.

Since $\{f_{k}\}_{k=1}^{\infty}$ is a Riesz sequence, $\{f_{k}\}_{k=1}^{\infty}$ and $\{S^{-1}(f_{k})\}_{k=1}^{\infty}$ are biorthogonal. Thus $U(f_{k})=g_{k}$ for all $k$.

By the assumption, for any $f\in L$, we have $$\|f-U(f)\|\leq \mu\|f\|.$$  Consequently, $$\|U(f)\|\geq (1-\mu)\|f\|.$$

Therefore, for any $\{c_{k}\}_{k=1}^{\infty}\in l^{2}$, we have
\begin{align*}
\|\sum_{k=1}^{\infty}c_{k}g_{k}\|&=\|U(\sum_{k=1}^{\infty}c_{k}f_{k})\|\\
&\geq (1-\mu)\|\sum_{k=1}^{\infty}c_{k}f_{k}\|\\
&\geq (1-\mu)\sqrt{A}(\sum_{k=1}^{\infty}|c_{k}|^{2})^{1\over{2}}.
\end{align*}

\end{proof}

\section{Perturbations of Schauder frames in Banach spaces}
\begin{thm}
Suppose that $(x_{n}, f_{n})_{n=1}^{\infty}$ is a (unconditional) Schauder frame of Banach
space $X$. Let $\{y_{n}\}_{n=1}^{\infty}$ be a sequence in $X$ which
satisfies $\mu:=\sum_{n=1}^{\infty}\|x_{n}-y_{n}\|\|f_{n}\|<1$.
Then there exists a sequence $\{g_{n}\}_{n=1}^{\infty}$ in $X^{*}$ such that
$(y_{n},g_{n})_{n=1}^{\infty}$ is a (unconditional) Schauder frame of $X$.
If, in addition, $\lambda:=\sum_{n=1}^{\infty}{{\|x_{n}-y_{n}\|}\over{\|x_{n}\|}}<\infty$,
then $Z_{Min}$ is a sequence space associated to $(y_{n},g_{n})_{n=1}^{\infty}$; namely,
the operator
$$U:Z_{Min}\rightarrow X,
U(\sum_{n=1}^{\infty}c_{n}e^{Min}_{n})=\sum_{n=1}^{\infty}c_{n}y_{n}$$ is
bounded with
$\|U\|\leq 1+\lambda$.
And
$$V:X\rightarrow Z_{Min},V(x)=\sum_{n=1}^{\infty}g_{n}(x)e_{n}^{Min}$$ is also a bounded operator which satisfies
$${1\over{(1+\lambda)}}\|x\|\leq \|V(x)\|\leq {K\over{1-\mu}}\|x\|,\forall x\in X,$$
where $K$ is the projection constant of $(x_{n},f_{n})_{n=1}^{\infty}$.

\end{thm}

\begin{proof}
Assume that $(x_{n}, f_{n})_{n=1}^{\infty}$ is a Schauder frame.
Define the operator $$L:X\rightarrow X, x\mapsto
\sum_{n=1}^{\infty}f_{n}(x)y_{n}.$$
By the hypothesis, an easy argument shows that $L$ is well-defined and
$\|I-L\|\leq \mu<1$.  This implies that $L$ is
invertible. Let $g_{n}=(L^{-1})^{*}f_{n}$ for each $n$. It is easy to check that
$(y_{n},g_{n})_{n=1}^{\infty}$ is a Schauder frame of $X$.

If $(x_{n}, f_{n})_{n=1}^{\infty}$ is an
unconditional Schauder frame, we'll show that $(y_{n},g_{n})_{n=1}^{\infty}$ is also unconditional. It suffices to show that,
for any $x\in X$ and for any increasing sequence $\{i_{k}\}_{k=1}^{\infty}$ of positive integers, the series $\sum_{k=1}^{\infty}g_{i_{k}}(x)y_{i_{k}}$ converges in norm. Indeed, for $m\leq n$, we have
\begin{align*}
\|\sum_{k=m}^{n}g_{i_{k}}(x)y_{i_{k}}\|&=\|\sum_{k=m}^{n}f_{i_{k}}(L^{-1}(x))y_{i_{k}}\|\\
&\leq \|\sum_{k=m}^{n}f_{i_{k}}(L^{-1}(x))(y_{i_{k}}-x_{i_{k}})\|+\|\sum_{k=m}^{n}f_{i_{k}}(L^{-1}(x))x_{i_{k}}\|\\
&\leq \sum_{k=m}^{n}\|f_{i_{k}}\|\|L^{-1}(x)\|\|y_{i_{k}}-x_{i_{k}}\|+\|\sum_{k=m}^{n}f_{i_{k}}(L^{-1}(x))x_{i_{k}}\|\\
&\rightarrow 0(m,n\rightarrow\infty).
\end{align*}
Hence $\sum_{k=1}^{\infty}g_{i_{k}}(x)y_{i_{k}}$ converges.

Now assume that $\sum_{n=1}^{\infty}{{\|x_{n}-y_{n}\|}\over{\|x_{n}\|}}<\infty$,
we'll show that $Z_{Min}$ is a sequence space associated to $(y_{n},g_{n})_{n=1}^{\infty}$.
First we prove that $U$ is well-defined and $\|U\|\leq1+\lambda$.
Indeed, for any $(c_{i})_{i}\in c_{00}$, we have
\begin{align*}
\|\sum_{i}c_{i}y_{i}\|&\leq
\|\sum_{i}c_{i}(x_{i}-y_{i})\|+\|\sum_{i}c_{i}x_{i}\|\\
&\leq \sum_{i}|c_{i}|\|x_{i}-y_{i}\|+\|\sum_{i}c_{i}e^{Min}_{i}\|_{Min}\\
&=\sum_{i}\|c_{i}x_{i}\|{{\|x_{i}-y_{i}\|}\over{\|x_{i}\|}}+\|\sum_{i}c_{i}e^{Min}_{i}\|_{Min}\\
&\leq (1+\lambda)\|\sum_{i}c_{i}e^{Min}_{i}\|_{Min}.
\end{align*}

Therefore $$\|U\|\leq 1+\lambda.$$

Finally, for each $x\in X$, one has
\begin{align*}
\|V(x)\|&=\|\sum_{n=1}^{\infty}g_{n}(x)e_{n}^{Min}\|\\
&=\|\sum_{n=1}^{\infty}f_{n}(L^{-1}(x))e_{n}^{Min}\|\\
&=\|T_{Min}(L^{-1}(x))\|\\
&\leq {K\over{1-\mu}}\|x\|.
\end{align*}
On the other hand,
$$\|x\|=\|UV(x)\|\leq (1+\lambda)\|V(x)\|.$$

The result follows.

\end{proof}

\begin{remark}
It is obvious that, in Theorem 3.1,  ``$\sum_{n=1}^{\infty}{{\|x_{n}-y_{n}\|}\over{\|x_{n}\|}}<\infty$'',
in general, does not imply  ``$\sum_{n=1}^{\infty}\|x_{n}-y_{n}\|\|f_{n}\|<1$''.
On the other hand, if $\{x_{n}\}_{n=1}^{\infty}$ is a basis for $X$ and
$\{f_{n}\}_{n=1}^{\infty}$ is the coefficient functionals associated to $\{x_{n}\}_{n=1}^{\infty}$, then  ``$\sum_{n=1}^{\infty}\|x_{n}-y_{n}\|\|f_{n}\|<1$'' implies ``$\sum_{n=1}^{\infty}{{\|x_{n}-y_{n}\|}\over{\|x_{n}\|}}<\infty$'' because $\{x_{n}\}_{n=1}^{\infty}$ and $\{f_{n}\}_{n=1}^{\infty}$ are biorthogonal. However, the following example shows that this is not true in the case of Schauder frames.
\end{remark}

\begin{example}
Let $\{e_{n}\}_{n=1}^{\infty}$ be an orthonormal basis for $\mathcal{H}$. Choose an integer $N$ with $\sum_{n=N+1}^{\infty}{1\over{n^{3\over{2}}}}<1$.
Let $\{x_{k}\}_{k=1}^{\infty}$ be the sequence where each ${e_{n}\over{n}}$ is repeated by $n$ times. That is,
$$\{x_{k}\}_{k=1}^{\infty}: e_{1}, {e_{2}\over{2}}, {e_{2}\over{2}}, {e_{3}\over{3}}, {e_{3}\over{3}}, {e_{3}\over{3}},...$$
Again let $\{f_{k}\}_{k=1}^{\infty}$ be the sequence where each $e_{n}$ is repeated by $n$ times. That is,
$$\{f_{k}\}_{k=1}^{\infty}: e_{1}, e_{2}, e_{2}, e_{3}, e_{3}, e_{3},...$$
Then $(x_{k},f_{k})_{k=1}^{\infty}$ is a Schauder frame for $\mathcal{H}$.

Let $a_{k}={1\over{(k+N)^{3\over{2}}}}(k=1,2,...)$ and take any unit element $e$ in $\mathcal{H}$. We let $$\{y_{k}\}_{k=1}^{\infty}=\{x_{k}+a_{k}e\}_{k=1}^{\infty}.$$  Then $$\sum_{k=1}^{\infty}\|x_{k}-y_{k}\|\|f_{k}\|<1.$$
But
\begin{align*}
\sum_{k=1}^{\infty}{{\|x_{k}-y_{k}\|}\over{\|x_{k}\|}}&=\sum_{k=1}^{\infty}{{a_{k}}\over{\|x_{k}\|}}\\
&=\sum_{n=1}^{\infty}n(\sum_{k={{n(n-1)}\over{2}}+1}^{{{n(n-1)}\over{2}}+n}a_{k})\\
&=\sum_{n=1}^{\infty}n({1\over ({{n(n-1)}\over{2}}+1+N)^{3\over{2}}}+...+{1\over ({{n(n-1)}\over{2}}+n+N)^{3\over{2}}})\\
&\geq \sum_{n=1}^{\infty}n\cdot n\cdot {1\over ({{n(n-1)}\over{2}}+n+N)^{3\over{2}}}\\
&=\infty.
\end{align*}

\end{example}

\begin{cor}
Let $(x_{k},f_{k})_{k=1}^{\infty}$ be a near-Schauder basis of Banach space $X$. If $\{y_{k}\}_{k=1}^{\infty}$ is a sequence in $X$ which satisfies the following two conditions:
\item[(1)]$\lambda:=\sum_{n=1}^{\infty}{{\|x_{n}-y_{n}\|}\over{\|x_{n}\|}}<\infty$;
\item[(2)]$\mu:=\sum_{n=1}^{\infty}\|x_{n}-y_{n}\|\|f_{n}\|<1$.

Then there exists a sequence $\{g_{k}\}_{k=1}^{\infty}$ in $X^{*}$ such that
$(y_{k},g_{k})_{k=1}^{\infty}$ is a near-Schauder basis; In this case,  $(x_{k},f_{k})_{k=1}^{\infty}$ and $(y_{k},g_{k})_{k=1}^{\infty}$ has the same excess.
\end{cor}
\begin{proof}
First it follows from Theorem 3.1 that there exists a sequence $\{g_{k}\}_{k=1}^{\infty}$ in $X^{*}$ such that
$(y_{k},g_{k})_{k=1}^{\infty}$ is a Schauder frame for $X$ and hence $\{y_{k}\}_{k=1}^{\infty}$ is complete in $X$. Then we shall show that $(y_{k},g_{k})_{k=1}^{\infty}$ is a near-Schauder basis.
Suppose that $(x_{k},f_{k})_{k=1}^{\infty}$ is a near-Schauder basis with excess $n$.
We may assume that, by changing the index set, $\{x_{k}\}_{k=1}^{\infty}=\{x_{k}\}_{k=1}^{n}\cup \{x_{k}\}_{k=n+1}^{\infty}$,
where $\{x_{k}\}_{k=n+1}^{\infty}$ is a Schauder basis for $X$. Then there exists a constant $C>0$ such that $$\|\sum_{k=n+1}^{s}c_{k}x_{k}\|\leq C\|\sum_{k=n+1}^{t}c_{k}x_{k}\|$$
for all scalars $c_{n+1},c_{n+2},...,c_{t}$ and for all $t\geq s\geq n+1$.

By (1), we choose $m>n$ such that $$2C\cdot \sum_{k=m+1}^{\infty}{{\|x_{k}-y_{k}\|}\over{\|x_{k}\|}}<1.$$
Define $$L:{\overline{span}}\{x_{k}\}_{m+1}^{\infty}\rightarrow {\overline{span}}\{y_{k}\}_{m+1}^{\infty}, L(\sum_{k=m+1}^{\infty}c_{k}x_{k})=\sum_{k=m+1}^{\infty}c_{k}y_{k}.$$
Then $L$ is well-defined and bounded. Indeed, $\sum_{k=m+1}^{\infty}c_{k}x_{k}\in {\overline{span}}\{x_{k}\}_{m+1}^{\infty}$.

Note that for each $k\geq m+1$, we have
$$|c_{k}|={{\|\sum_{j=m+1}^{k}c_{j}x_{j}-\sum_{j=m+1}^{k-1}c_{j}x_{j}\|}\over{\|x_{k}\|}}
\leq {{2C\cdot \|\sum_{j=m+1}^{\infty}c_{j}x_{j}\|}\over{\|x_{k}\|}}.$$

Then for all $t\geq s\geq m+1$, one has
\begin{align*}
\|\sum_{k=s}^{t}c_{k}y_{k}\|&\leq \|\sum_{k=s}^{t}c_{k}(y_{k}-x_{k})\|+\|\sum_{k=s}^{t}c_{k}x_{k}\|\\
&\leq \sum_{k=s}^{t}|c_{k}|\|y_{k}-x_{k}\|+\|\sum_{k=s}^{t}c_{k}x_{k}\|\\
&\leq 2C\cdot \|\sum_{j=m+1}^{\infty}c_{j}x_{j}\|\sum_{k=s}^{t}{{\|y_{k}-x_{k}\|}\over{\|x_{k}\|}}+\|\sum_{k=s}^{t}c_{k}x_{k}\|\\
&\rightarrow 0(s,t\rightarrow \infty).
\end{align*}
Hence $\sum_{k=m+1}^{\infty}c_{k}y_{k}$ converges and $\|L\|\leq 2C\lambda+1$.

Moreover
\begin{align*}
\|\sum_{k=m+1}^{\infty}c_{k}x_{k}-L(\sum_{k=m+1}^{\infty}c_{k}x_{k})\|&=\|\sum_{k=m+1}^{\infty}c_{k}x_{k}-\sum_{k=m+1}^{\infty}c_{k}y_{k}\|\\
&\leq \sum_{k=m+1}^{\infty}|c_{k}|\|x_{k}-y_{k}\|\\
&\leq 2C\|\sum_{j=m+1}^{\infty}c_{j}x_{j}\|\cdot \sum_{k=m+1}^{\infty}{{\|x_{k}-y_{k}\|}\over{\|x_{k}\|}}.
\end{align*}
This implies that $L({\overline{span}}\{x_{k}\}_{m+1}^{\infty})$ is closed.

Since $L(x_{k})=y_{k}$ for each $k\geq m+1$, we have
$L({\overline{span}}\{x_{k}\}_{m+1}^{\infty})={\overline{span}}\{y_{k}\}_{m+1}^{\infty}$
and $\{y_{k}\}_{k=m+1}^{\infty}$ is a basic sequence.

Again by Theorem 8 in \cite{CK}, one has $$dim(X/{\overline{span}}\{x_{k}\}_{m+1}^{\infty})=dim(X/{\overline{span}}\{y_{k}\}_{m+1}^{\infty}).$$
As in the proof of Corollary 2.2, by the completeness of $\{y_{k}\}_{k=1}^{\infty}$ in $X$,
we can pick the same numbers of linearly independent elements from $\{x_{k}\}_{k=1}^{m}$(resp.
$\{y_{k}\}_{k=1}^{m}$) to add to $\{x_{k}\}_{k=m+1}^{\infty}$(resp.
$\{y_{k}\}_{k=m+1}^{\infty}$) to get Schauder bases for $X$. This proves that $\{x_{k}\}_{k=1}^{\infty}$ and $\{y_{k}\}_{k=1}^{\infty}$ has the same excess.

\end{proof}

The following result is also a Schauder frame version of another perturbation theorem of bases in Banach spaces.

\begin{thm}
Suppose that $(x_{n}, f_{n})_{n=1}^{\infty}$ is a Schauder frame of  Banach
space $X$. Let $\{y_{n}\}_{n=1}^{\infty}$ be a sequence in $X$ which
satisfies
$\lambda:=\sum_{n=1}^{\infty}{{\|x_{n}-y_{n}\|}\over{\|x_{n}\|}}<{1\over{K}}$,where $K$ is the projection constant of $(x_{n},f_{n})_{n=1}^{\infty}$.
Then there exists a sequence $\{g_{n}\}_{n=1}^{\infty}$ in $X^{*}$ such that
$(y_{n},g_{n})_{n=1}^{\infty}$ is a Schauder frame of $X$. Moreover, $Z_{Min}$ is a sequence space associated to $(y_{n},g_{n})_{n=1}^{\infty}$, namely,
the operator
$$U:Z_{Min}\rightarrow X,
U(\sum_{n=1}^{\infty}c_{n}e^{Min}_{n})=\sum_{n=1}^{\infty}c_{n}y_{n}$$ is
bounded with
$\|U\|\leq
1+\lambda.$

and
$$V:X\rightarrow Z_{Min},V(x)=\sum_{n=1}^{\infty}g_{n}(x)e_{n}^{Min}$$ is also a bounded operator which satisfies
$${1\over{1+\lambda}}\|x\|\leq \|V(x)\|\leq {K\over{1-\lambda\cdot K}}\|x\|,\forall x\in X.$$
In addition, if $(x_{n}, f_{n})_{n=1}^{\infty}$ is an
unconditional Schauder frame, then $(y_{n},g_{n})_{n=1}^{\infty}$ is also an unconditional Schauder frame.
\end{thm}

\begin{proof} Similar to Theorem 3.1, we can derive that $U$ is well-defined and $\|U\|\leq
1+\lambda$.

Define $$L:X\rightarrow X, x\mapsto
\sum_{n=1}^{\infty}f_{n}(x)y_{n}.$$ Then $L$ is well-defined. Indeed,
for $x\in X$ and for $m\leq n$, we have
\begin{align*}
\|\sum_{k=m}^{n}f_{k}(x)y_{k}\|&\leq \|\sum_{k=m}^{n}f_{k}(x)(y_{k}-x_{k})\|+\|\sum_{k=m}^{n}f_{k}(x)x_{k}\|\\
&\leq \sum_{k=m}^{n}\|f_{k}(x)x_{k}\|{{\|y_{k}-x_{k}\|}\over{\|x_{k}\|}}+\|\sum_{k=m}^{n}f_{k}(x)x_{k}\|\\
&\leq K\cdot \|x\|\sum_{k=m}^{n}{{\|y_{k}-x_{k}\|}\over{\|x_{k}\|}}+\|\sum_{k=m}^{n}f_{k}(x)x_{k}\|\\
&\rightarrow 0(m,n\rightarrow\infty).
\end{align*}
Therefore the series $\sum_{n=1}^{\infty}f_{n}(x)y_{n}$ converges and $\|L\|\leq \lambda\cdot K+1$.

Furthermore
\begin{align*}
\|x-L(x)\|&=\|x-\sum_{n=1}^{\infty}f_{n}(x)y_{n}\|\\
&=\|\sum_{n=1}^{\infty}f_{n}(x)(x_{n}-y_{n})\|\\
&\leq \sum_{n=1}^{\infty}|f_{n}(x)|\|x_{n}-y_{n}\|\\
&=\sum_{n=1}^{\infty}\|f_{n}(x)x_{n}\|{{\|x_{n}-y_{n}\|}\over{\|x_{n}\|}}\\
&\leq \lambda\cdot K\|x\|.
\end{align*}
Thus $$\|I-L\|\leq \lambda\cdot K<1.$$  As an operator, $L$ is
invertible. Let $g_{n}=(L^{-1})^{*}f_{n}$ for each $n$. Then
$(y_{n},g_{n})_{n=1}^{\infty}$ is a Schauder frame of $X$.

Subsequently, we compute the bounds of $\|V(x)\|(x\in X)$.
For each $x\in X$, one has
\begin{align*}
\|V(x)\|&=\|\sum_{n=1}^{\infty}g_{n}(x)e_{n}^{Min}\|\\
&=\|\sum_{n=1}^{\infty}f_{n}(L^{-1}(x))e_{n}^{Min}\|\\
&=\|T_{Min}(L^{-1}(x))\|\\
&\leq {K\over{1-\lambda\cdot K}}\|x\|.
\end{align*}
On the other hand,
$$\|x\|=\|UV(x)\|\leq (1+\lambda)\|V(x)\|.$$
Putting these two inequalities together, we have
$${1\over{1+\lambda}}\|x\|\leq \|V(x)\|\leq {K\over{1-\lambda\cdot K}}\|x\|.$$
Finally, if $(x_{n}, f_{n})_{n=1}^{\infty}$ is an unconditional Schauder frame, we want to show that $(y_{n}, g_{n})_{n=1}^{\infty}$ is also unconditional.
It is enough to prove that
for any $x\in X$ and for any increasing sequence $\{i_{k}\}_{k=1}^{\infty}$ of positive integers, the series $\sum_{k=1}^{\infty}g_{i_{k}}(x)y_{i_{k}}$ converges in norm. Indeed, for $m\leq n$, we have
\begin{align*}
\|\sum_{k=m}^{n}g_{i_{k}}(x)y_{i_{k}}\|&=\|\sum_{k=m}^{n}f_{i_{k}}(L^{-1}(x))y_{i_{k}}\|\\
&\leq \|\sum_{k=m}^{n}f_{i_{k}}(L^{-1}(x))(y_{i_{k}}-x_{i_{k}})\|+\|\sum_{k=m}^{n}f_{i_{k}}(L^{-1}(x))x_{i_{k}}\|\\
&\leq \sum_{k=m}^{n}\|f_{i_{k}}(L^{-1}(x))x_{i_{k}}\|\cdot{{\|y_{i_{k}}-x_{i_{k}}\|}\over{\|x_{i_{k}}\|}}+\|\sum_{k=m}^{n}f_{i_{k}}(L^{-1}(x))x_{i_{k}}\|\\
&\leq K\cdot \|L^{-1}(x)\|\sum_{k=m}^{n}{{\|y_{i_{k}}-x_{i_{k}}\|}\over{\|x_{i_{k}}\|}}+\|\sum_{k=m}^{n}f_{i_{k}}(L^{-1}(x))x_{i_{k}}\|\\
&\rightarrow 0(m,n\rightarrow\infty)
\end{align*}
from which the conclusion follows and the proof is completed.

\end{proof}

Similarly, we have the following corollary.

\begin{cor}
Let $(x_{k},f_{k})_{k=1}^{\infty}$ be a near-Schauder basis of Banach space $X$. If $\{y_{k}\}_{k=1}^{\infty}$ is a sequence in $X$ which satisfies the condition in Theorem 3.3, then there exists a sequence $\{g_{k}\}_{k=1}^{\infty}$ in $X^{*}$ such that
$(y_{k},g_{k})_{k=1}^{\infty}$ is a near-Schauder basis; In this case,  $(x_{k},f_{k})_{k=1}^{\infty}$ and $(y_{k},g_{k})_{k=1}^{\infty}$ has the same excess.
\end{cor}

Finally, in this section, we consider the Schauder frames in dual spaces.

\begin{thm}
Suppose that $(x_{n}, f_{n})_{n=1}^{\infty}$ is an unconditional Schauder frame of a reflexive Banach
space $X$. Let $\{g_{n}\}_{n=1}^{\infty}$ be a sequence in $X^{*}$ which satisfies $$\mu:=\sum_{n=1}^{\infty}\|f_{n}-g_{n}\|\cdot \|x_{n}\|<1.$$
Then there is a sequence $\{y_{n}\}_{n=1}^{\infty}$ in $X$ so that
$(y_{n},g_{n})_{n=1}^{\infty}$ is an unconditional Schauder frame of $X$.
\end{thm}
\begin{proof}
First we show that $(f_{n}, x_{n})_{n=1}^{\infty}$ is a Schauder frame of $X^{*}$. Indeed, fix any $f\in X^{*}$. Since the sequence $\{\sum_{i=1}^{n}f(x_{i})f_{i}\}_{i=1}^{n}$ converges to $f$ in the $w^{*}-$topology, it suffices to prove that the series $\sum_{n=1}^{\infty}f(x_{n})f_{n}$ converges in norm. By The Orlicz-Pettis Theorem, we only need to show that for any increasing sequence $\{i_{k}\}_{k=1}^{\infty}$ of positive integers, the series $\sum_{k=1}^{\infty}f(x_{i_{k}})f_{i_{k}}$ converges weakly. Since $X$ is reflexive, it is enough to show that, for each $x\in X$, the sequence $$\{\sum_{k=1}^{n}f(x_{i_{k}})f_{i_{k}}(x)\}_{n=1}^{\infty}=\{f(\sum_{k=1}^{n}f_{i_{k}}(x)x_{i_{k}})\}_{n=1}^{\infty}$$ is convergent. Since $\sum_{n=1}^{\infty}f_{n}(x)x_{n}$ converges unconditionally in norm, the series $\sum_{k=1}^{\infty}f_{i_{k}}(x)x_{i_{k}}$ converges in norm. Consequently,
the sequence $\{\sum_{k=1}^{n}f(x_{i_{k}})f_{i_{k}}(x)\}_{n=1}^{\infty}$ converges.

Define $$T: X^{*}\rightarrow X^{*}, T(f)=\sum_{n=1}^{\infty}f(x_{n})g_{n}.$$
Then $T$ is well-defined,linear and bounded. Indeed, for $m\leq n$, by the assumption, we have
\begin{align*}
\|\sum_{i=m}^{n}f(x_{i})g_{i}\|&\leq \|\sum_{i=m}^{n}f(x_{i})(g_{i}-f_{i})\|+\|\sum_{i=m}^{n}f(x_{i})f_{i}\|\\
&\leq \sum_{i=m}^{n}\|f\|\cdot \|x_{i}\|\cdot \|g_{i}-f_{i}\|+\|\sum_{i=m}^{n}f(x_{i})f_{i}\|\\
&\rightarrow 0(m,n\rightarrow\infty).
\end{align*}
Moreover, $\|T\|\leq \mu+1$.
It is easy to check that  $\|I-T\|\leq \mu<1$.

Thus $T$ is an isomorphism from $X^{*}$ onto $X^{*}$. It follows from the reflexivity of $X$ that there exists an isomorphism $L$ from $X$ onto $X$ such that $L^{*}=T$.

Let $y_{n}=L^{-1}(x_{n})$ for each $n$. We finish the proof by showing that $(y_{n},g_{n})_{n=1}^{\infty}$ is an unconditional Schauder frame of $X$.
First we prove that the series $\sum_{n=1}^{\infty}g_{n}(x)y_{n}$ converges unconditionally in norm for any $x\in X$. Indeed, for any increasing sequence $\{i_{k}\}_{k=1}^{\infty}$ of positive integers and
for $m\leq n$, we have
\begin{align*}
\|\sum_{i=m}^{n}g_{i_{k}}(x)y_{i_{k}}\|&=\|L^{-1}(\sum_{i=m}^{n}g_{i_{k}}(x)x_{i_{k}})\|\\
&\leq \|L^{-1}\|\cdot (\|\sum_{i=m}^{n}(g_{i_{k}}(x)-f_{i_{k}}(x))x_{i_{k}}\|+\|\sum_{i=m}^{n}f_{i_{k}}(x)x_{i_{k}}\|)\\
&\leq \|L^{-1}\|\cdot (\sum_{i=m}^{n}\|g_{i_{k}}-f_{i_{k}}\|\|x_{i_{k}}\|\|x\|+\|\sum_{i=m}^{n}f_{i_{k}}(x)x_{i_{k}}\|)\\
&\rightarrow 0(m,n\rightarrow \infty).
\end{align*}
Thus the series $\sum_{k=1}^{\infty}g_{i_{k}}(x)y_{i_{k}}$ converges and hence $\sum_{n=1}^{\infty}g_{n}(x)y_{n}$ converges unconditionally.

So it remains to show that
$$f(x)=f(\sum_{n=1}^{\infty}g_{n}(x)y_{n}), \forall f\in X^{*}.$$
Indeed,
\begin{align*}
f(x)&=(T(T^{-1}f))(x)\\
&=(\sum_{n=1}^{\infty}(T^{-1}f)(x_{n})g_{n})(x)\\
&=(\sum_{n=1}^{\infty}((L^{-1})^{*}f)(x_{n})g_{n})(x)\\
&=(\sum_{n=1}^{\infty}f(L^{-1}(x_{n}))g_{n})(x)\\
&=(\sum_{n=1}^{\infty}f(y_{n})g_{n})(x)\\
&=\sum_{n=1}^{\infty}f(y_{n})g_{n}(x)\\
&=f(\sum_{n=1}^{\infty}g_{n}(x)y_{n}).
\end{align*}
The proof is completed.
\end{proof}

{\bf Acknowledgements.} We would like to express our gratitude to the referee's suggestions and comments.



\begin{thebibliography}{MMM}

\bibitem{B} R. P. Boas, {\em General expansion theorems}, Proc. Nat. Acad. Sci. U.S.A., {\bf 26}(1940),139-143.

\bibitem{CC1} P. G. Casazza and O. Christensen, {\em Perturbations of operators and applications to frame theory}, J. Fourier. Anal. Appl.
{\bf 3}(1997),543-557.

\bibitem{CC2} P. G. Casazza and O. Christensen, {\em The reconstruction property in Banach spaces and a perturbation theorem}, Canad. Math. Bull. {\bf 3}(
2008),348-358.

\bibitem{CDOSZ} P. G. Casazza, S. J. Dilworth, E. Odell, Th. Schlumprecht and A. Zsak, {\em Coefficient quantization for frames in Banach spaces}, J. Math. Anal. Appl.{\bf 348}(2008),66-86.


\bibitem{CK} P. G. Casazza and N. J. Kalton, {\em Generalizing the Paley-Wiener perturbation theory for Banach spaces},
Proc. Amer. Math. Soc.,{\bf 127}(1999), 519-527.


\bibitem{C1} O. Christensen, {\em A Paley-Wiener thoerem for frames },
Proc. Amer. Math. Soc.,{\bf 123}(1995),2199-2201.


\bibitem{C2} O. Christensen, {\em Frame perturbations},
Proc. Amer. Math. Soc.,{\bf 123}(1995),1217-1220.

\bibitem{C3} O. Christensen, {\em An introduction to frames and Riesz bases,} Birkh${\ddot{a}}$user,2003.

\bibitem{CH} O. Christensen and C. Heil, {\em Perturbations of Banach frames and atomic decompositions}, Math. Nachr.,{\bf 185}(1997),33-47.


\bibitem{CLL} O. Christensen, C. Lennard and C. Lewis, {\em Perturbation of frames for a subspace of a Hilbert space}, Rocky Mountain J. Math.,
{\bf 30}(2000),1237-1249.

\bibitem{FZ95} S. J. Favier and R. A. Zalik, {\em On the stability of frames and Riesz bases },
Applied and Computational Harmonic Analysis, {\bf 2}(1995), 160--173.

\bibitem{K} T. Kato, {\em Perturbation theory for linear operators}, Springer, New York, 1976.


\bibitem{PW} R. E. A. C. Paley and N. Wiener, {\em Fourier transforms in the complex domain}, Amer. Math. Soc. Coll.
Publ.{\bf 19}, New York, 1934.


\bibitem{S} I. Singer, {\em Bases in Banach spaces I}, Springer-Verlag, Berlin, 1970.


\end{thebibliography}
\end{document}